\documentclass[12pt]{article}
\usepackage {graphics} 
\usepackage {graphicx}  
\usepackage {pstricks,pst-node,  pst-coil} 
\onecolumn  
\usepackage{amsmath,amssymb,latexsym}
\usepackage{amsmath,amsthm}
\newtheorem{tw}{Theorem}  
\newtheorem{lem}[tw]{Lemma}  
  
\newtheorem{cnj}[tw]{Conjecture}  
\newtheorem{cor}[tw]{Corollary}  
\newtheorem{rem}[tw]{Remark}

\begin{document}
\hyphenation{every}
\title{Local Irregularity Conjecture for 2-multigraphs versus cacti}
\author{Igor Grzelec\thanks{AGH University of Science and Technology, al. A. Mickiewicza 30, $30-059$  Krak\'ow, Poland, grzelec@agh.edu.pl}, Mariusz Woźniak\thanks{AGH University of Science and Technology, al. A. Mickiewicza 30, $30-059$  Krak\'ow, Poland, mwozniak@agh.edu.pl Research project partially supported by program "Excellence initiative-research university" for the AGH University of Science and Technology.}}
\maketitle
\begin{abstract}
A multigraph is locally irregular if the degrees of the end-vertices of every multiedge are distinct.  The \textit{locally irregular coloring} is an edge coloring of a multigraph $G$ such that every color induces a locally irregular submultigraph of $G$. A \textit{locally irregular colorable} multigraph $G$ is any multigraph which admits a locally irregular coloring. We denote by ${\rm lir}(G)$ the \textit{locally irregular chromatic index} of a multigraph $G$, which is the smallest number of colors required in the locally irregular coloring of the locally irregular colorable multigraph $G$. In case of graphs the definitions are similar. The Local Irregularity Conjecture for 2-multigraphs claims that for every connected graph $G$, which is not isomorphic to $K_2$, multigraph $^2G$ obtained from $G$ by doubling each edge satisfies ${\rm lir}(^2G)\leq 2$. We show this conjecture for cacti. This class of graphs is important for the Local Irregularity Conjecture for 2-multigraphs and the Local Irregularity Conjecture which claims that every locally irregular colorable graph $G$ satisfies ${\rm lir}(G)\leq 3$. At the beginning it has been observed that all not locally irregular colorable graphs are cacti. Recently it has been proved that there is only one cactus which requires 4 colors for a locally irregular coloring and therefore  the Local Irregularity Conjecture was disproved.
\\

\textit{Keywords:} locally irregular coloring; decomposable; cactus graphs; 2-multigraphs.
\end{abstract}
\section{Introduction}
All graphs and multigraphs considered in this paper are finite. The main interest of this paper are edge colorings of a multigraph. Let $G=(V, E)$ be a graph. We say that a graph is \textit{locally irregular} if the degrees of the two end-vertices of every edge are distinct. A \textit{locally irregular coloring} of a graph $G$ is an edge coloring of $G$ such that every color induces a locally irregular subgraph of $G$. We denote by ${\rm lir}(G)$ the \textit{locally irregular chromatic index} of a graph $G$ which is the smallest number $k$ such that there exists a locally irregular coloring of $G$ with $k$ colors. This problem is closely related to the well known 1-2-3 Conjecture proposed by Karoński, Łuczak and Thomason in \cite{Karonski Luczak Thomason}.

Note that not every graph has a locally irregular coloring. We define the family $\mathfrak{T}$ recursively in the following way:
\begin{itemize}
  \item the triangle $K_3$ belongs to $\mathfrak{T}$,
  \item if $G$ is a graph from $\mathfrak{T}$, then any graph $G'$ obtained from $G$ by identifying a vertex $v\in V(G)$ of degree 2, which belongs to a triangle in $G$, with an end vertex of a path of even length or with an end vertex of a path of odd length such that the other end vertex of that path is identified with a vertex of a new triangle.
\end{itemize}
The family $\mathfrak{T'}$ consists of the family $\mathfrak{T}$, all odd length paths and all odd length cycles. In \cite{Baudon Bensmail Przybylo Wozniak} Baudon, Bensmail, Przybyło and Woźniak proved that only the graphs from the family $\mathfrak{T'}$ are not locally irregular colorable (do not admit locally irregular coloring). They also proposed the Local Irregularity Conjecture which says that every connected graph $G \notin \mathfrak{T'}$ satisfies ${\rm lir}(G)\leq 3$. 

However, Sedlar and \v Skrekovski in \cite{Sedlar Skrekovski} showed that the bow-tie graph $B$ which is presented in Figure \ref{bow-tie graph} do not have locally irregular coloring with three colors.
\begin{figure}[h!]
\centering
\includegraphics[width=4.5cm]{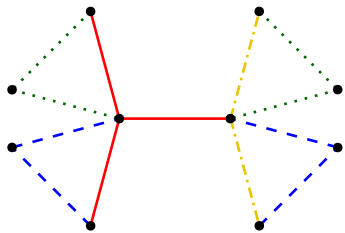}
\caption{The bow-tie graph $B$ and its locally irregular coloring with four colors.}
\label{bow-tie graph}
\end{figure}
They also proved in \cite{Sedlar Skrekovski 2} that every locally irregular colorable cactus $G \neq B$ satisfies ${\rm lir}(G)\leq 3$. By a \textit{cactus} we mean a graph in which no two cycles intersect in more than one vertex. Furthermore, they proposed the following new version of the Local Irregularity Conjecture.

\begin{cnj}[\cite{Sedlar Skrekovski 2}]
\label{graph3}
Every connected graph $G \notin \mathfrak{T'}$ except for the bow-tie graph $B$ satisfies ${\rm lir}(G)\leq 3$.
\end{cnj}

This conjecture was proved for some graph classes for example trees \cite{Baudon Bensmail}, graphs with minimum degree at least $10^{10}$ \cite{Przybylo}, $r$-regular graphs where \linebreak $r\geq 10^7$ \cite{Baudon Bensmail Przybylo Wozniak}, decomposable split graphs \cite{Lintzmayer Mota Sambinelli} and decomposable claw-free graphs with maximum degree 3 \cite{Luzar Macekova}. For planar graphs it is known that ${\rm lir}(G)\leq 15$ \cite{Bensmail Dross Nisse}. For general connected graphs Bensmail, Merker and Thomassen \cite{Bensmail Merker Thomassen} proved that ${\rm lir}(G)\leq 328$ if $G\notin \mathfrak{T'}$.  Later the constant upper bound was lowered to ${\rm lir}(G)\leq 220$ by Lu\v zar, Przybyło and Soták \cite{Luzar Przybylo Sotak}.

Before we present the Local Irregularity Conjecture for 2-multigraphs we present more definitions. By 2-\textit{multigraph}, denoted by $^2G$, we mean the multigraph obtained from a graph $G$ by doubling each edge. We call an \textit{edge multiplicity} the number of single edges forming a multiedge $e$ in a multigraph $G$ and we denote it by $\mu_G(e)$. Multigraph $H$ is a \textit{submultigraph} of a multigraph $G$ if $H$ is a subgraph of $G$ and for each multiedge $e$ of $H$, $\mu_H(e) \leq \mu_G(e)$ holds. Analogically, multigraph $H$ is an \textit{induced submultigraph} of $G$ if $H$ is an induced subgraph of $G$ and for each multiedge $e$ of $H$, $\mu_H(e)= \mu_G(e)$ holds. We denote by $\hat d(v)$ degree of the vertex $v$ in a multigraph (the number of single edges incident with the vertex $v$). The \textit{locally irregular coloring} of a multigraph $G$ is an edge coloring of $G$ such that every color induces a locally irregular submultigraph of $G$. We say that multigraphs $\hat G_1$ and $\hat G_2$ create a \textit{decomposition} of a multigraph $\hat G$ if for each edge $e$ of $G$, $\mu_{\hat G_1}(e)+ \mu_{\hat G_2}(e)= \mu_{\hat G}(e)$ holds. We say that a multigraph is \textit{locally irregular colorable} if it satisfies the locally irregular coloring. By the \textit{locally irregular chromatic index} of a locally irregular colorable multigraph $G$, denoted by ${\rm lir}(G)$, we mean the smallest number of colors required in a locally irregular coloring of $G$. We will say that a multiedge is colored red-blue if one element of the multiedge is red and the second is blue. Grzelec and Woźniak proposed in \cite{Grzelec wozniak} the following Conjecture for 2-multigraphs.
\begin{cnj}[Local Irregularity Conjecture for 2-multigraphs \cite{Grzelec wozniak}]
\label{main}
For every connected graph $G$ which is not isomorphic to $K_2$ we have ${\rm lir}(^2G)\leq 2$.
\end{cnj}

They also proved that for general connected graphs ${\rm lir}(^2G)\leq 76$ if $G$ is not isomorphic to $K_2$. Moreover, they proved the following theorem concerning the above conjecture.

\begin{tw}[\cite{Grzelec wozniak}]
\label{cycle}
The Local Irregularity Conjecture for $2$-multigraphs holds for: paths, cycles, wheels $W_n$, complete graphs $K_n$, for $n \geq 3$, complete \linebreak k-partite graphs and bipartite graphs. $\quad \quad \quad \quad \quad \quad \quad \quad \quad \quad \quad \quad \quad \quad \quad \quad \Box$
\end{tw}

Now, we recall the locally irregular coloring of multipaths and multicycles from the proof of above theorem given by Grzelec and Woźniak in \cite{Grzelec wozniak} which will be useful throughout the proof of our main result. We will denote by $P_n$ a path with $n$ vertices.

For a multipath $^2P_n$ of even length we color the first two multiedges blue, the next two multiedges red and we repeat this color sequence until the end of the multipath. We do not consider the multipath of length one. For a longer multipath of odd length we color the first multiedge blue, the second red-blue, the third red and then we color the remaining multiedges in the same way as multipath of even length.

\begin{figure}[h!]
\centering
\includegraphics[width=1\textwidth]{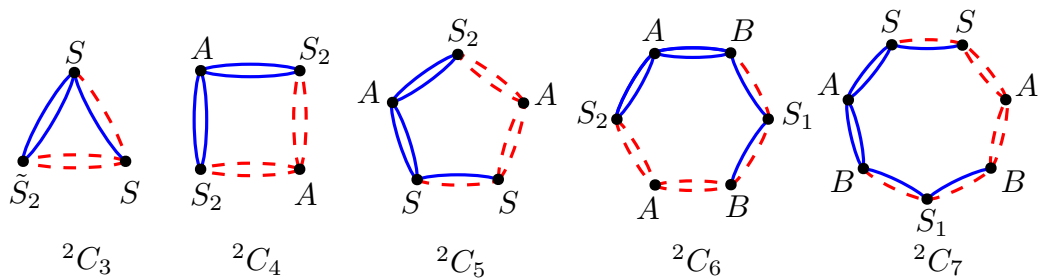}
\caption{A locally irregular coloring of the multicycles. The meaning of vertex labels will be explained in Theorem \ref{cactus}.}
\label{fCycles1}
\end{figure}

We color multicycles of length from three to seven as in Figure \ref{fCycles1}. The coloring of longer multicycle we obtain by adding multipath of length divisible by four colored in the same way as above to the appropriate colored multicycle of length from four to seven after two red multiedges.

In this paper we will show Conjecture \ref{main} holds for cacti. The mentioned results for graphs so far indicate that cacti are important class of graphs for the problem of locally irregular coloring, because the not locally irregular colorable graphs (the family $\mathfrak{T}$) are cacti and the only known counterexample for the Local Irregularity Conjecture is also a cactus. This motivated us to check the Local Irregularity Conjecture for 2-multigraphs for cacti.

\section{Main result}
Before we prove Conjecture \ref{main} for cacti we need to introduce the notation and lemma for trees which will be useful through the proof. First, we will denote by $T(r)$ a tree rooted at a vertex $r$. A \textit{shrub} is any tree rooted at a leaf. The only edge in a shrub $G$ incident to the root we will call the  \textit{root edge} of $G$.  We will call \textit{branch} rooted at vertex $x$, denoted by $B(x)$, the following subgraph in rooted tree $T$. If $x$ has only one son $x_0$ in $T$ a branch $B(x)$ includes: edge $xx_0$, edges from $x_0$ to all its sons $x_1, \dots, x_k$ and if $x_j$ has only one neighbor $x'_j$ which is a leaf in $T$  edge  $x_jx'_j$ for $j=1, \dots, k$. If $x$ has more than one son in $T$ a branch $B(x)$ includes: edges from $x$ to all its sons $x_1, \dots, x_k$ and if $x_j$ has only one neighbor $x'_j$ which is a leaf in $T$  edge  $x_jx'_j$  for $j=1, \dots, k$. We can easily see that any branch contain at least two edges. 
\begin{rem}
There are only three branches: path $P_4$ rooted at an internal vertex, path $P_4$ rooted at an initial vertex and path $P_5$ rooted at a central vertex, which are not locally irregular.
\end{rem}

We will use corresponding notation for 2-multigraphs.

\begin{lem}
\label{l2}
For each tree $T$ rooted at the vertex $r$ there exists a decomposition of $^2T$ generated by a red-blue almost locally irregular coloring where the only possible conflict may occur between $r$ and one or two of its neighbours. 
\end{lem}
\begin{proof}
The lemma holds obviously for $T=K_2$. Assume that $T\ne K_2$. First, we decompose $^2T(r)$ into branches in the following way. Let $^2B(r)$ be the first branch from our decomposition. Let $x_i$ for $i=1, \dots, p$ be a leaf in $^2B(r)$, which is not a leaf in $T$. The next branches of the decomposition are $^2B(x_1), \dots, ^2B(x_p)$. We start coloring from $^2B(r)$, next we color branches with root at a vertex from $^2B(r)$ and so on. Branch $^2B(r)$ has always one color except for the situation when $^2B(r)=$ $^2T(r)$ and $^2B(r)$ is a multipath $^2P_4$ rooted at one of its ends (exception I). In this situation we color root multiedge blue and the rest of this branch red. Let $^2B(x)$ where $x\ne r$ be a branch from the above decomposition. If $^2B(x)$ is not an exceptional branch, we color $^2B(x)$ using the color different from the color of multiedge from the vertex $x$ to its father. 

If $^2B(x)$ is an exceptional branch, we denote by $x_0$ the father of $x$. Assume that multiedge $x_0x$ is colored red. We use the coloring of exceptional branch $^2B(x)$ presented in Figure \ref{fex1}.
\begin{figure}[h!]
\centering
\includegraphics[width=12cm]{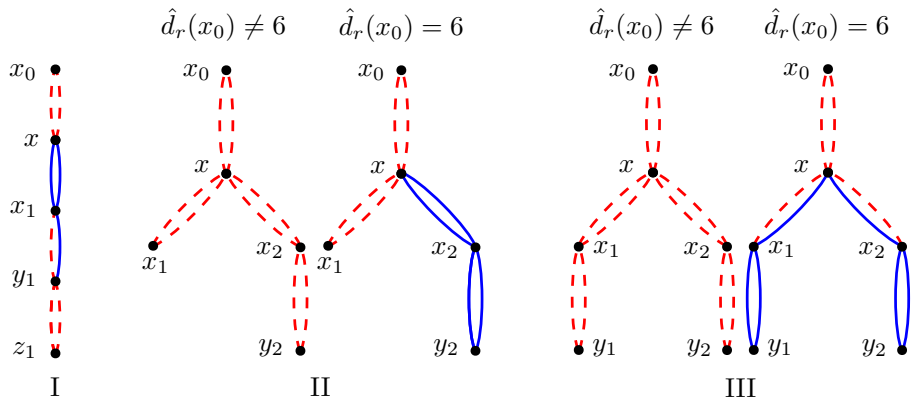}
\caption{The coloring of branch $^2B(x)$ which is an exception.}
\label{fex1}
\end{figure}
Note that multiedge $x_0x$ has always only one color.
\end{proof}

Remark that in the coloring of multitree rooted at vertex $r$ presented in the proof of above lemma, the root $r$ and its sons are always even. 

We need one more definition. In multicactus we call \textit{cyclic} all vertices on multicycles and the remaining vertices are \textit{woody}. Now we are ready to prove our main result for cacti. 

\begin{tw}
\label{cactus}
For every cactus $G$, the multigraph $^2G$ satisfies ${\rm lir}(^2G)\leq 2$.
\end{tw}

\begin{proof}
We give a construction of locally irregular coloring of $^2G$. We will use the method of vertex labeling. We treat multicactus as a multigraph obtained from a multicycle by \textit{adding} to vertices on this multicycle elements (multitrees or multicycles) another subsequent multicycles to leaves at multitrees, next subsequent multitrees and multicycles to vertices on added multicycles and so on. By adding or joining rooted multitree $^2T(r)$ to a vertex $x$ we mean identifying the vertex $x$ with the root of multitree $^2T(r)$. Similar, by adding or joining a multicycle to a vertex $x$ we mean identifying the vertex $x$ with one of the vertices on the multicycle. 

Our construction consists of three main steps and then we repeat steps two and three until we color the whole multigraph $^2G$. In each step we extend existing red-blue locally irregular coloring to added elements. Moreover, each cyclic vertex is labelled with $A$, $B$, $S_1$, $S_2$, $\tilde S_2$, $S$, $S'$, $T$ or $T'$ according to the following rules. We label by:  
\begin{itemize}
  \item $A$ each even vertex whose neighbors on each multicycle have labels from the set $\{B, S_2, S, S'\}$ and has red degree or blue degree greater than two,
  \item $B$ each odd vertex whose neighbors on each multicycle have labels from the set $\{A, S_1, T, T'\}$,
  \item $S_1$ even vertex $x$ with: $\hat d(x)=4$, $\hat d_r(x)=\hat d_b(x)=2$ and two incident multiedges colored red-blue, 
  \item $S_2$ even vertex $x$ on multicycle of length greater than three with: \linebreak $\hat d(x)=4$, $\hat d_r(x)=\hat d_b(x)=2$ and two incident multiedges colored first blue second red,
  \item $\tilde S_2$ even vertex  $x$ on $^2C_3$ with: $\hat d(x)=4$, $\hat d_r(x)=\hat d_b(x)=2$ and two incident multiedges colored first blue second red,
  \item $S$ both vertices from a special pair of odd adjacent vertices, 
  \item $S'$ both vertices from a pair of adjacent vertices which would become the special pair $S$ if we joined something to both vertices from this pair, because now this pair has only one neighbour which has label $A$ or all neighbours of this pair are labelled $A$ or $A$ and $S_1$ and pair $S'$ has form: multiedge colored blue, first vertex from pair $S'$, multiedge colored red, second vertex from pair $S'$ and multiedge colored red-blue (or symmetrically),
  \item $T$ both vertices from a special pair of even adjacent vertices on $^2C_3$ (on longer multicycle we never use this label),    
  \item $T'$ both vertices from a pair of adjacent vertices on $^2C_3$ which would become the special pair $T$ if we joined something to both vertices from this pair, because now it consist of the vertex labelled $A$ or $\tilde S_2$ and the vertex labelled $B$ (on longer multicycle we never use this label).
\end{itemize} 
\textbf{Step 1: \quad initial part of the locally irregular coloring of $\mathbf{^2G}$.} In this step we color in a standard way, except for the situation described below, the longest multicycle in $^2G$ using the same method as in the proof of Theorem \ref{cycle} (see Figure \ref{fCycles1}). During the rest of the proof we will often use this coloring of multicycle. Note that multitree $^2P_3$ rooted at its end creates particular problems. We will call such multipath a \textit{spike}. Indeed, it cannot be added neither to the vertex labelled $S_1$ nor to the vertex labelled $\tilde S_2$.  We can avoid this problem in two ways. First, trying to change colors in standard coloring of multicycle so that the vertex to which we add spike has another label. Sometimes it is impossible, for example in the situation when multicycle $^2C_3$ has added a spike to each vertex. This particular case is presented in Figure \ref{C3 with spikes}.  
\begin{figure}[h!]
\centering
\includegraphics[width=5cm]{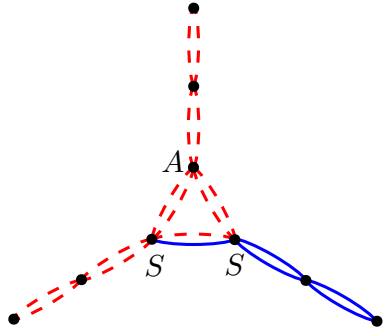}
\caption{The coloring of $^2C_3$ with spikes.}
\label{C3 with spikes}
\end{figure}
Therefore, in the situation when the initial multicycle has length $4k+2$ or $4k+3$ for $k\geq 1$ and the next step requires adding a spike we will consider the multicycle with spike instead of multicycle.  The coloring of this multicycle with spike is presented in Figure \ref{C6 with spike}. Note that in this figure we presented coloring only for multicycles of length $n=6$ and $n=7$, but we can easily extend this coloring for longer multicycles with spike.  
\begin{figure}[h!]
\centering
\includegraphics[width=9.5cm]{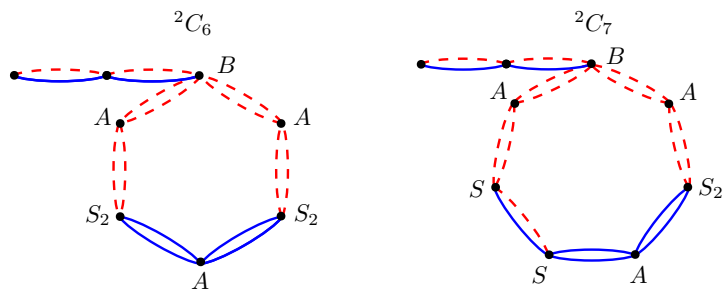}
\caption{The initial coloring of $^2C_6$ with spike and $^2C_7$ with spike.}
\label{C6 with spike}
\end{figure}

We may assume that in the next step we will never add a spike to the vertex labeled $\tilde S_2$ on $^2C_3$. If it is not true we can easily recolor standard $^2C_3$. \\
\textbf{Step 2:\quad joining all multitrees and multicycles to vertices on chosen colored multicycle in $\mathbf{^2G}$.} We choose one colored multicycle in $^2G$. First, we join all multitrees and multicycles to vertices labeled $A$, $B$, $S_1$, $S_2$ and $\tilde S_2$ on chosen multicycle. Then we consider pair of vertices labelled $S$ or $S'$ or $T$ or $T'$. When we consider a colored multicycle $^2C_5$ with pair of vertices labelled $S'$ we start from this pair to avoid potential problems with neighbors of this pair. Note that in any colored multicycle there is at most one pair of such vertices. When we consider each of these pairs, we join at the same time all multitrees and all multicycles to both vertices creating this pair. Below we present details of each part of Step 2. 

Let $\tilde G$ be already colored part of the multigraph $^2G$. This means in particular that each cyclic vertex in $\tilde G$ has its own label. Let $x$ be a cyclic vertex in $\tilde G$. In this part of Step 2 we describe how to add subsequent elements to $x$. The method of coloring and labelling them depends mostly on the label of $x$. We establish a rule that if we start coloring elements added to the vertex $x$ then we color all those elements. In the next steps we shall not add anything to $x$. We assume that multicycles which we join are colored standardly and labelled as in Figure \ref{fCycles1}.

\textbf{Joining to the vertex labelled $\mathbf{A}$.} Note that we do not change the parity of the vertex labelled $A$. Therefore, adding multicycles is very simple in this case. By a dominating color at a vertex $x$ we mean color in which the vertex $x$ has greater degree. We join to the vertex $x$ labelled $A$ multicycles:
\begin{itemize}
  \item $^2C_3$ using their vertices labelled  $\tilde S_2$,
  \item $^2C_{4k}$ and $^2C_{4k+1}$ using their vertices labelled $A$ so that we increase the degree of $x$ at dominating color,
  \item $^2C_{4k+2}$ and $^2C_{4k+3}$ using their vertices labelled $S_1$.
\end{itemize}
At the end, we join rooted multitree. So we join a rooted multitree which is not multishrub except for $^2K_2$ and $^2P_3$ to the vertex $x$ labelled $A$ starting from multiedges incident to its root colored with the dominating color at $x$ using Lemma \ref{l2}. If the rooted multitree joined to the vertex $x$ labelled $A$ is a multishrub and is not isomorphic to $^2K_2$ or $^2P_3$ we color its root multiedge with the dominating color at $x$ and the rest of this 2-shrub rooted at the vertex $r'$ we color starting from the other color using Lemma \ref{l2}. If $\hat d(r')=6$ in $^2G$ and we have pendant multipath $^2P_3$ from the vertex $r'$, we recolor this multipath red-blue. 

\textbf{Joining to the vertex labelled $\mathbf{B}$.}
Note that a vertex labelled $B$ is odd and remains odd. Therefore, adding multitrees is very simple in this case because from Lemma \ref{l2} we always have $\hat d(x)\ne \hat d(r')$ for each $r'$ which is a neighbour of $x$ in added multitree. When adding multicycle, we identify a vertex $y$ with the vertex $x$. We need to make sure that neighbours of $y$ on the multicycle have label $A$ or $T$. For multicycle  $^2C_{4k}$ or $^2C_{4k+1}$ we use standard coloring of a multicycle and the vertex $y$ with label $S_2$. For multicycle $^2C_{4k+2}$ or $^2C_{4k+3}$ except for $^2C_3$ we denote by $y$ the vertex with label $B$ and greater degree at dominating color at $x$. Then we recolor incident multiedges to the vertex labelled $S_1$ on joined multicycle on dominating color at $y$. Thus, the vertex $x$ remains odd. For a multicycle $^2C_3$ we use standard coloring of a multicycle and the vertex $y$ with label $\tilde S_2$. Then we recolor second multiedge red on joined $^2C_3$. Thus, the neighbors of $y$ on $^2C_3$ create a special pair of type $T$.

\textbf{Joining to the vertex labelled $\mathbf{S_1}$ or $\mathbf{\tilde S_2}$.}
We will never join an individual spike to the vertex labelled $S_1$ or $\tilde S_2$. We join all multicycles to the vertex $x$ labelled $S_1$ or $\tilde S_2$ using the same method as when we join to the vertex labelled $A$. Then we join multitree to the vertex $x$ using the method of joining to the vertex labelled $A$. Note that after joining all the elements the vertex $x$ gets label $A$.

\textbf{Joining to the vertex labelled $\mathbf{S_2}$.} Note that in this case we change the parity and label of the vertex $x$. If we add  multicycles then we use multicycle to change parity of the vertex $x$ and then we add remaining elements using the same method as when we add to the vertex labelled $B$. Let $y$ be the vertex labelled $B$ on the multicycle of length equal to two or three modulo four except for $^2C_3$ such that its neighbour $y_1$ is labelled $S_1$. We join this multicycle to $x$ using the vertex $y$. If we have a spike at vertex $y_1$, we replace this multicycle by a multicycle with spike. The coloring of this multicycle with spike is presented in Figure \ref{long multicycle with spike}.
\begin{figure}[h!]
\centering
\includegraphics[width=10.5cm]{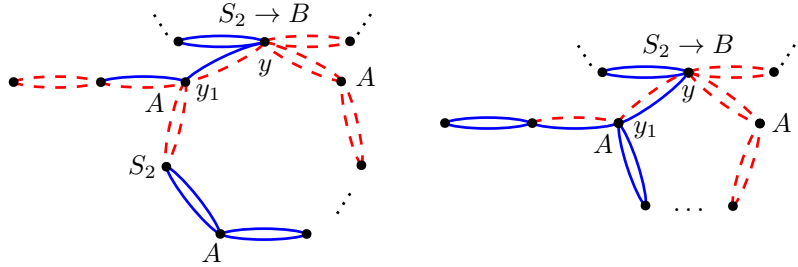}
\caption{The coloring of a multicycle $^2C_{4k+2}$ or $^2C_{4k+3}$ with spike joined to the vertex labelled $S_2$ (left) and $^2C_{4k}$ or $^2C_{4k+1}$ with spike joined to the vertex labelled $S_2$ (right)}
\label{long multicycle with spike}
\end{figure}

A multicycle of length equal to zero or one modulo four in the standard coloring either has no odd vertex or has an odd vertex with an odd neighbour. Therefore, this coloring cannot be used in this case. Thus, before we join multicycle $^2C_{4k}$ to the vertex $x$ labelled $S_2$ we recolor two multiedges incident to one of the vertices labelled $A$ with red-blue. Thus, this vertex $y_1$ gets label $S_1$ and its neighbours get label $B$. Then we join this recolored multicycle $^2C_{4k}$ to the vertex $x$ using similar method as when we join $^2C_{4k+2}$ or $^2C_{4k+3}$ to the vertex labelled $S_2$. Note that we do not have any pair of vertices labelled $S'$ or $S$ on joined multicycle $^2C_{4k}$.

Before we join multicycle $^2C_{4k+1}$ except for $^2C_5$ to $x$, we recolor two multiedges incident to one of the vertices labelled $A$ which has no neighbour from the special pair $S$ red-blue. Thus, this vertex $y_1$ gets label $S_1$ and its neighbours get label $B$. Then we join this recolored multicycle $^2C_{4k+1}$ to $x$ using the same method as when we join $^2C_{4k}$ to the vertex labelled $S_2$.

The coloring of multicycle $^2C_5$,  $^2C_5$ with spike and $^2C_3$ joined to the vertex labelled $S_2$ is presented in Figure \ref{short multicycles}. We always use presented coloring of $^2C_5$ with spike when we join $^2C_5$ to the vertex labelled $S_2$ and at least one of the neighbours of $y$ on joined $^2C_5$ has a spike, because we would like to avoid potential problems with pair $S'$. 
\begin{figure}[h!]
\centering
\includegraphics[width=\textwidth]{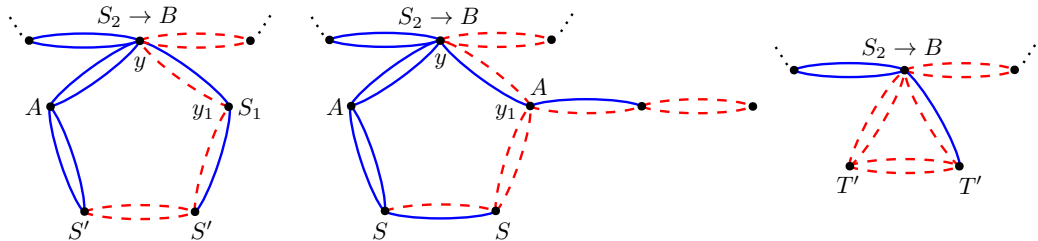}
\caption{The coloring of multicycle $^2C_5$, $^2C_5$ with spike and $^2C_3$  joined to the vertex labelled $S_2$ on multicycle.}
\label{short multicycles}
\end{figure}

Now, we introduce the method of joining the only multitree $^2T$ rooted at $r$ to the vertex $x$ labelled $S_2$. We consider two cases. Assume that red is the dominating color at the vertex $x$. \\
\textbf{Case 1:}\quad $\hat d(r)=2$ in $^2T$. We color root multiedge $rx_0$ red-blue. If $\hat d(x_0)>4$, we color multitree rooted at vertex $x_0$ using Lemma \ref{l2} so that blue is the dominating color at the vertex $x_0$. If $\hat d(x_0)=4$, we color multiedge $x_0x_1$ red-blue and then we color multitree rooted at the vertex $x_1$ using Lemma \ref{l2} so that blue is the dominating color at the vertex $x_1$.  \\
\textbf{Case 2:}\quad $\hat d(r)>2$ in $^2T$. We treat this rooted multitree as the set of multishrubs with common root. We choose one of those multishrubs and color it in the same way as in the first case. Then we color the remaining part of multitree using the same method as when we join it to the vertex labelled $B$ so that red is the dominating color at $x$.

Now, we present in details method of joining elements to pairs of vertices labelled $T$, $T'$, $S'$ and $S$ on the multicycle. Recall that pair of vertices labelled $T$ or $T'$ will appear only on the multicycle $^2C_3$. Below in Figure \ref{pairs} we give the form of each pair but it may happen that we will have pairs with symmetrical coloring and then we treat those pairs analogously.  
\begin{figure}[h!]
\centering
\includegraphics[width=12cm]{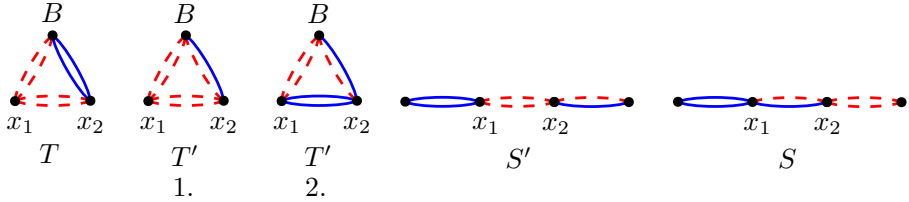}
\caption{The form of pair vertices labelled $T$, $T'$, $S'$ and $S$.}
\label{pairs}
\end{figure}

\textbf{Pair T.}\quad When we join to only one vertex from the special pair $T$ we choose the vertex $x_1$ and we use the method of joining to the vertex labeled $A$. When we join $^2K_2$ to one vertex and something  to the second vertex from par $T$ we join blue $^2K_2$ to $x_2$ and remaining elements to $x_1$ using the method of joining to the vertex labeled $A$. The particular coloring of special pair $T$ with added a spike to each vertex from this pair is presented in Figure \ref{2-spike}.
\begin{figure}[h!]
\centering
\includegraphics[width=4.5cm]{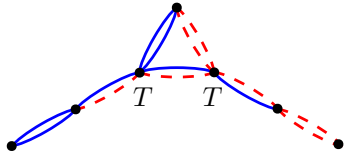}
\caption{The particular coloring of special pair $T$ with spikes.}
\label{2-spike}
\end{figure}

We consider the situation when we will join something (except for a spike) to one vertex from pair $T$ and a spike to the second vertex from pair $T$. We join a spike colored red to the vertex $x_1$ and then we join elements to the vertex $x_2$ starting from multicycles and then possibly multitree using the same method as when we join to the vertex labelled $A$ and so that blue is the dominating color at $x_2$. If we have a conflict ($\hat d_r(x_1)= \hat d_r(x_2)=6$), we recolor a spike joined to the vertex $x_1$ blue. 

In the remaining situation we join starting from multicycles then we join multitrees to vertices $x_1$ and $x_2$ using the same method as when we join to the vertex labelled $A$ so that red is the dominating color at $x_1$ and blue at $x_2$. If we have conflict between $x_1$ and $x_2$ without recoloring, we move all joins from $x_1$ to $x_2$ and from $x_2$ to $x_1$. Thus, we are done with special pair $T$.  

\textbf{Pair T'.}\quad  First, we present the general method of joining multitree $^2T$ rooted at $r$ to the vertex $x_2$. We consider two cases.\\
\textbf{Case 1:}\quad $\hat d(r)=2$ in $^2T$. We color the root multiedge $rx_0$ red-blue. Then we color the multitree rooted at the vertex $x_0$ using Lemma \ref{l2}. \\
\textbf{Case 2:}\quad $\hat d(r)\geq4$ in $^2T$. We treat the multitree rooted at $r$ as a set of multishrubs with common root. We choose one multishrub rooted at $r$ and color it in the same way as in Case 1. Then we join the remaining part of $^2T$ to the vertex $x_2$ using the method of joining to the vertex labelled $A$ starting from the root multiedge in the dominating color at $x_2$ if $\hat d(r)=4$ in $^2T$ and starting from the other color of multiedges incident to $r$ in the remaining part of $^2T$ if $\hat d(r)>4$ in $^2T$.

Now, we present in detail method of joining elements to the pair $T'$ in the first form. When we join to only one vertex from pair $T'$, we choose the vertex $x_1$ and we use the method of joining to the vertex labelled $A$.

When we join only a multitree to each vertex from pair $T'$, we join the first multitree to $x_1$ using the same method as when we join to the vertex labelled $A$ and the second multitree to the vertex $x_2$ using the general method of joining multitree rooted at $r$ to the vertex $x_2$ so that blue is the dominating color at $x_2$ if we join multitree with root of degree equal at least six. If we join multitree with root of degree equal to two and multitree with root of degree equal to four we join multitree with root of degree equal to two to the vertex $x_2$ and multitree with root of degree equal to four to the vertex $x_1$ to avoid conflict.

In the remaining case we first join exactly one multicycle to the vertex $x_2$ using the coloring presented in Figure \ref{short multicycles2}. The coloring of longer multicycle we obtain by adding multipath of length divisible by four colored standardly (first two multiedges blue, next two multiedges red, and so on) to the appropriate colored multicycle of length from four to seven after two red multiedges.
\begin{figure}[h!]
\centering
\includegraphics[width=\textwidth]{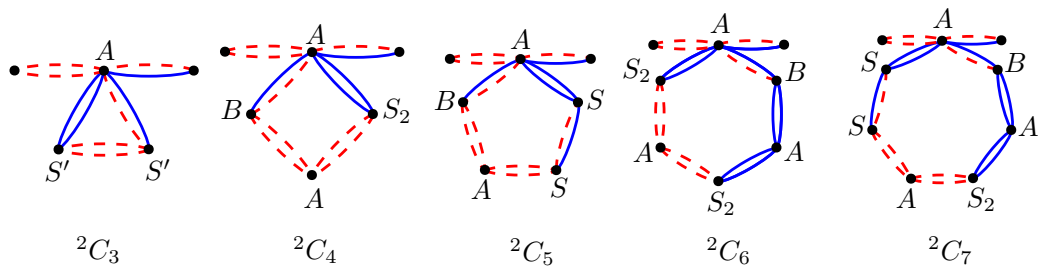}
\caption{The coloring of multicycle joined at first to the vertex $x_2$ from pair $T'$.}
\label{short multicycles2}
\end{figure}
Then we continue joining multicycles and possibly multitree to the vertex $x_2$ using the same method as when we join to the vertex labelled $A$ so that blue is the dominating color at $x_2$. Next, we join elements to the vertex $x_1$ using the same method as when we join to the vertex labelled $A$ so that red is the dominating color at $x_1$. At the end if we have conflict between vertices $x_1$ and $x_2$, namely $\hat d_r(x_1)=\hat d_r(x_2)$, we recolor symmetrically (we change color of all multiedges) joined at first multicycle to the vertex $x_2$. 

Now, we present the method of joining elements to the pair $T'$ in the second form. When we join to only one vertex from pair $T'$,
we choose the vertex $x_2$ and consider two situations. In the first situation we join only multitree using the general method of joining multitree to the vertex $x_2$. In the second situation when we join multicycle or multicycles and possibly multitree we use the method of joining to the vertex $x_2$ form pair $T'$ in the first form. When we join something to both vertices from pair $T'$ in the second form we start with recoloring multiedge $x_1x_2$ red. Thus, pair $T'$ get the first form and we use the above method of joining to pair $T'$ in the first form. Thus, we are done with pair $T'$.

\textbf{Pair S'.}\quad This pair occurs only on multicycle $^2C_3$ and $^2C_5$. Additionally, if the pair $S'$ occurs on a multicycle $^2C_5$ joined to the colored part of multicactus at vertex $x$,
then neighbors of the vertex $x$ on $^2C_5$ have no spike and the distance between $x$ and $x_1$ is the same as the distance between $x$ and $x_2$. By the above properties of pair $S'$ and the fact that we start considering a colored $^2C_5$ which contain a pair $S'$ with this pair we can choose where we join (to $x_1$ or to $x_2$) the first and the second part of elements to this pair $S'$.

When we join to only one vertex from pair $S'$, we choose vertex $x_2$ and we use the method of joining to the vertex labelled $B$, so that red is the dominating color at $x_2$. Then we can also join rooted multitree to the vertex $x_1$ using the same method as when we join to the vertex labelled $S_2$ so that blue is the dominating color at $x_1$. 

We consider other case when we join at least one multicycle to each vertex from pair $S'$. First, we join all from the first part of joins to this pair $S'$ to the vertex $x_2$ using the method of joining to the vertex labelled $B$, so that red is the dominating color at $x_2$. Then we join exactly one multicycle to $x_1$ using the method of joining to the vertex labelled $S_2$, so that $\hat d_r(x_1)=3$ and $\hat d_b(x_2)=5$. Next, we join multicycles and possibly multitree to the vertex $x_1$ using the same method as when we join to the vertex labelled $B$ so that blue is the dominating color at $x_1$. At the end if we have conflict between vertices $x_1$ and $x_2$, namely $\hat d_r(x_1)=\hat d_r(x_2)$, we recolor symmetrically (we change color of all multiedges) joined at first multicycle to the vertex $x_1$. Thus, we are done with pair $S'$.

\textbf{Pair S.}\quad We consider two cases. \\
\textbf{Case 1:}\quad we join only one multicycle $^2C_3$ or $^2C_{4k}$ or $^2C_{4k+1}$ to the special pair $S$. When we join $^2C_3$ we use one of the following colorings presented in Figure \ref{Pair S}.
\begin{figure}[h!]
\centering
\includegraphics[width=10cm]{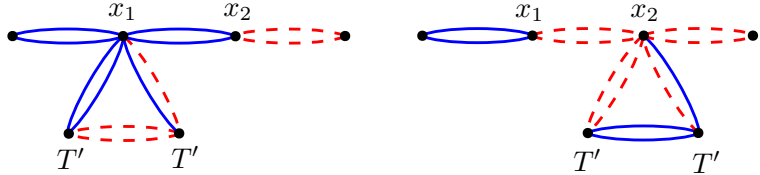}
\caption{Colorings of multicycle $^2C_3$ joined to the special pair $S$.}
\label{Pair S} 
\end{figure}

When we join multicycle $^2C_{4k}$ or $^2C_{4k+1}$ to the vertex $x_1$ we first recolor multiedge $x_1x_2$ blue and then we join this multicycle using the method of joining to the vertex labelled $S_2$. When we join multicycle $^2C_{4k}$ or $^2C_{4k+1}$ to the vertex $x_2$ we first recolor multiedge $x_1x_2$ red and then we join this multicycle using the method of joining to the vertex labelled $S_2$. \\
\textbf{Case 2:}\quad otherwise, we join rooted multitree and all kinds of multicycles to the vertex $x_1$ using the method of joining to the vertex labelled $B$ so that blue is the dominating color at $x_1$. Then we join rooted multitree and all kinds of  multicycles to the vertex $x_2$ using the method of joining to the vertex labelled $B$ so that red is the dominating color at $x_2$. Note that red multiedges at $x_1$ come from multicycles $^2C_{4k}$, $^2C_{4k+1}$ and $^2C_3$, similarly, blue multiedges at the vertex $x_2$ come from multicycles $^2C_{4k}$, $^2C_{4k+1}$ and $^2C_3$.

Obviously, we can have conflict between vertices $x_1$ and $x_2$. First, we consider conflict when one vertex from pair $S$ suppose $x_1$ admits $\hat d(x_1)=4$. It means that $\hat d_b(x_1)=\hat d_b(x_2)=3$. Therefore, it suffices that we recolor symmetrically one of the multicycles $^2C_{4k+3}$, $^2C_{4k+2}$ or rooted multitree joined to the vertex $x_2$ to solve this conflict. If we have similar conflict and $x_2$ admits $\hat d(x_2)=4$ we solve it analogously. Assume that $\hat d(x_1)>4$, $\hat d(x_2)>4$ and we have conflict between $x_1$ and $x_2$. Thus, we have definitely joined at least two multicycles from the set of all multicycles $^2C_{4k+1}$, $^2C_{4k}$ and $^2C_3$ to the vertex $x_1$ or $x_2$. To solve this conflict we recolor using the coloring from the method of joining to the vertex labelled $S_2$ and so that we increase the degree in the dominating color exactly two multicycles from the set of all multicycles $^2C_{4k+1}$, $^2C_{4k}$ and $^2C_3$ joined to one vertex from pair $S$. Thus, we always solve this conflict, because using above procedure we increase the degree in the dominating color by two and decrease the degree in the other color by two in one vertex from pair $S$. Thus, we are done with pair $S$.\\
\textbf{Step 3:\quad joining multicycles and possibly multitree to a leaf in the colored part of $\mathbf{^2G}$.} Note that we never join a single multitree to a leaf in $^2G$. We denote by $x$ chosen leaf in colored part of $^2G$ and by $x_0$ the only neighbour of $x$ in colored part of $^2G$. We present the method of joining the longest multicycle to the leaf. After joining this multicycle vertex $x$ gets label $A$ or $B$ and has different parity than $x_0$. Therefore, we can easily join remaining elements to $x$. We will consider four main cases in view of the color of multiedge $xx_0$ and parity of $x_0$ degrees in both colors.  \\
\textbf{Case 1:}\quad multiedge $xx_0$ is monochromatic, without loss of generality red, and $x_0$ is even. First, we describe the method of joining $^2C_3$ to the vertex $x$. When we join $^2C_3$ and something else to $x$ we use the first coloring of $^2C_3$ presented in Figure \ref{Step3a}. Note that if we join also the only multitree to $x$ we should tend to make blue the dominating color at $x$ to avoid potential conflict. In other case when we join only $^2C_3$ to $x$ we use the second coloring of $^2C_3$ presented in Figure \ref{Step3a}.
\begin{figure}[h!]
\centering
\includegraphics[width=5cm]{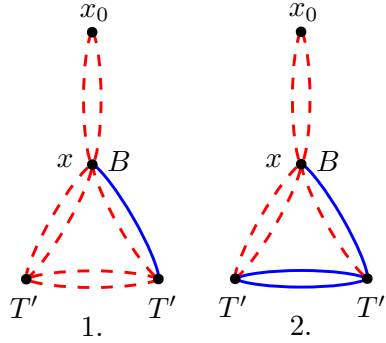}
\caption{Colorings of multicycle $^2C_3$ joined to the vertex $x$ in Case 1.}
\label{Step3a} 
\end{figure}

We join multicycle of length greater than three to the vertex $x$ using the method of joining to the vertex labelled $S_2$. Thus, in all this case the vertex $x$ with joined multicycle has label $B$. \\
\textbf{Case 2:}\quad multiedge $xx_0$ is monochromatic, without loss of generality red, and $x_0$ is odd. We join multicycle to the vertex $x$ using the method of joining to the vertex labelled $A$ so that red is the dominating color at $x$. Thus, the vertex $x$ with joined multicycle has label $A$.\\
\textbf{Case 3:}\quad multiedge $xx_0$ is colored red-blue and $x_0$ is even. We join multicycle to the vertex $x$ using the method of joining to the vertex labelled $B$ so that red is the dominating color at $x$. Thus, the vertex $x$ with joined multicycle has label $B$.\\
\textbf{Case 4:}\quad multiedge $xx_0$ is colored red-blue and $x_0$ is odd. We join a multicycle to the vertex $x$ using the coloring presented in Figure \ref{Step3d}. The coloring of longer multicycle we obtain by adding multipath of length divisible by four colored standardly (first two multiedges blue, next two multiedges red, and so on) to the appropriately colored multicycle of length from four to seven after two red multiedges.
\begin{figure}[h!]
\centering
\includegraphics[width=\textwidth]{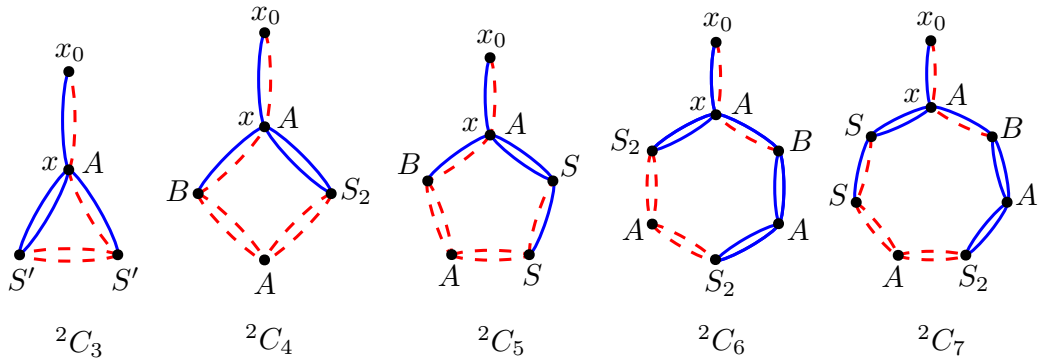}
\caption{The coloring of multicycle joined to the vertex $x$ in Case 4.}
\label{Step3d} 
\end{figure}
Thus, the vertex $x$ with joined multicycle has label $A$. So we are done in this step.
\end{proof}

As an immediate consequence of the above theorem and Theorem \ref{cycle} we get the following result.

\begin{cor}
The Local Irregularity Conjecture for $2$-multigraphs holds for graphs from the family $\mathfrak{T'}$.
\end{cor}


\begin{thebibliography}{99}
\bibitem{Baudon Bensmail}O. Baudon, J. Bensmail, \'E. Sopena, \textit{On the complexity of determining the irregular chromatic index of a graph}, J. Discret. Algorithms {\bf 30} (2015) 113 -- 127.
\bibitem{Baudon Bensmail Przybylo Wozniak}O. Baudon, J. Bensmail, J. Przybyło, M. Woźniak, \textit{On decomposing regular graphs into locally irregular subgraphs}, European Journal of Combinatorics {\bf 49} (2015), 90–104.
\bibitem{Bensmail Dross Nisse}J. Bensmail, F. Dross, N. Nisse, \textit{Decomposing degenerate graphs into locally irregular subgraphs}, Graphs Comb. {\bf 36} (2020), 1869-1889.
\bibitem{Bensmail Merker Thomassen}J. Bensmail, M. Merker, C. Thomassen, \textit{Decomposing graphs into a constant number of locally irregular subgraphs}, European Journal of
Combinatorics {\bf 60} (2017), 124–134.
\bibitem{Grzelec wozniak}I. Grzelec, M. Woźniak, \textit{On decomposing multigraphs into locally irregular submultigraphs},  available at https://arxiv.org/pdf/2208.08809.pdf
\bibitem{Karonski Luczak Thomason}M. Karoński, T.  Łuczak, A. Thomason, \textit{Edge weights and vertex colours}, J. Combin. Theory Ser. B {\bf91(1)} (2004), 151-157.
\bibitem{Lintzmayer Mota Sambinelli} C.N. Lintzmayer, G.O. Mota, M. Sambinelli, \textit{Decomposing split graphs into locally irregular graphs}, Discrete Appl. Math. {\bf 292} (2021), 33–44.
\bibitem{Luzar Macekova}B. Lužar, M. Maceková, S. Rindošová, R. Soták, K. Sroková, K. Štorgel, \textit{Locally irregular edge-coloring of subcubic graphs}, available at https://arxiv.org/pdf/2210.04649.pdf
\bibitem{Luzar Przybylo Sotak}B. Lu\v zar, J. Przybyło, R. Soták, \textit{New bounds for locally irregular chromatic index of bipartite and subcubic graphs}, Journal of Combinatorial
Optimization {\bf 36(4)} (2018), 1425–1438.
\bibitem{Przybylo}J. Przybyło, \textit{On decomposing graphs of large minimum degree into locally irregular subgraphs}, Electron. J. Combin. {\bf 23}  (2016), 2-31.
\bibitem{Sedlar Skrekovski 2}J. Sedlar R. \v Skrekovski, \textit{Local Irregularity Conjecture vs. cacti}, available at https://arxiv.org/pdf/2207.03941.pdf
\bibitem{Sedlar Skrekovski}J. Sedlar R. \v Skrekovski, \textit{Remarks on the Local Irregularity Conjecture}, Mathematics {\bf9(24)} (2021), 3209. https://doi.org/10.3390/math9243209
\end{thebibliography}
\end{document}